\documentclass{amsart}

\usepackage{color,graphicx,amssymb,latexsym,amsfonts,txfonts,amsmath,amsthm}
\usepackage{pdfsync}
\usepackage{amscd}
\usepackage[all,cmtip]{xy}

\usepackage{hyperref}
\hypersetup{
    colorlinks=true,       
    linkcolor=blue,          
    citecolor=blue,        
    filecolor=blue,      
    urlcolor=blue           
}

\input epsf

\newtheorem{theo}{Theorem}
\newtheorem{coro}{Corollary}

\newtheorem{lemm}{Lemma}

\theoremstyle{remark}
\newtheorem{rema}{\bf Remark}
\newtheorem{example}{\bf Example}

\begin{document}

\title{Quadrangular ${\mathbb Z}_{p}^{l}$-actions on Riemann surfaces}
\author{Ruben A. Hidalgo}

\subjclass[2010]{30F10, 14H37, 14H30, 14H40}
\keywords{Riemann surface; Automorphisms; Algebraic curves; Jacobian variety}

\address{Departamento de Matem\'atica y Estad\'{\i}stica, Universidad de La Frontera. Temuco, Chile}
\email{ruben.hidalgo@ufrontera.cl}
\thanks{Partially supported by Projects Fondecyt 1190001 and 1220261}

\begin{abstract}
Let $p \geq 3$ be a prime integer and, for $l \geq 1$, let $G \cong {\mathbb Z}_{p}^{l}$ be a group of conformal automorphisms of some closed Riemann surface $S$ of genus $g \geq 2$. By the Riemann-Hurwitz formula, either $p \leq g+1$ or $p=2g+1$. 
If $l=1$ and $p=2g+1$, then $S/G$ is the sphere with exactly three cone points and, if moreover $p \geq 7$, then $G$ is the unique $p$-Sylow subgroup of ${\rm Aut}(S)$. If $l=1$ and $p=g+1$, then $S/G$ is the sphere with exactly four cone points and, if moreover $p \geq 13$, then $G$ is again the unique $p$-Sylow subgroup. The above unique facts permited many authors to obtain algebraic models and the corresponding groups ${\rm Aut}(S)$ in these situations.
Now, let us assume $l \geq 2$. If $p \geq 5$, then either (i) $p^{l} \leq g-1$ or (ii) $S/G$ has genus zero, $p^{l-1}(p-3) \leq 2(g-1)$ and $2 \leq l \leq r-1$, where $r \geq 3$ is the number of cone points of $S/G$. Let us assume we are in case (ii).  If $r=3$, then $l=2$ and $S$ happens to be the classical Fermat curve of degree $p$, whose group of automorphisms is well known. The next case, $r=4$, is studied in this paper. We provide an algebraic curve representation for $S$, a description of its group of conformal automorphisms, a discussion of its field of moduli and an isogenous decomposition of its jacobian variety.
\end{abstract}

\maketitle

\section{Introduction}
If $S$ is a closed Riemann surface of genus $g \geq 2$, then Schwarz \cite{Schwarz} proved that ${\rm Aut}(S)$, the group of 
conformal automorphisms of $S$, is finite and later Hurwitz \cite{Hurwitz} obtained the upper bound  $|{\rm Aut}(S)| \leq 84(g-1)$.  Moreover, if $G$ is a finite group, then Greenberg \cite{Greenberg} proved the existence of a closed Riemann surface $S$, of some genus $g \geq 2$, such that $G<{\rm Aut}(S)$. The 
moduli space ${\mathfrak M}_{g}$, of biholomorphism classes of closed Riemann surfaces of genus $g \geq 2$, is a complex orbifold of dimension $3(g-1)$ \cite{Nag}. Its 
branch locus ${\mathcal B}_{g} \subset {\mathcal M}_{g}$ consists of those (classes of) Riemann surfaces with non-trivial groups of automorphisms. 
If $g \geq 4$, then ${\mathcal B}_{g}$ coincides with the locus where ${\mathfrak M}_{g}$ fails to be a topological manifold.

As a consequence of the Riemann-Roch theorem, in one direction, and the implicit function theorem in the other, there is an equivalence between the category of closed Riemann surfaces (up to biholomorphisms) and that of irreducible complex algebraic curves (up to birationality). In such an equivalence, conformal maps between closed Riemann surfaces corresponds to rational maps between the curves. Unfortunately, in the general case, it is a difficult problem to provide an algebraic curve for $S$ and the correspoding algebraic presentation of $G$. Only some particular situations have been worked out in the literature (for instance, hyperelliptic Riemann surfaces, cyclic $p$-gonal curves, Hurwitz curves, Fermat curves, generalized Fermat curves).  

Let $S$ be a closed Riemann surface of genus $g \geq 2$ and let $G<{\rm Aut}(S)$ be such that $S/G$ has genus zero (so, by the uniformization theorem, it can be identified with the Riemann sphere $\widehat{\mathbb C}$) and it has $r \geq 3$ conical points; we  will say for short that $G$ is of type $r$. 
In this case, up to a M\"obius transformation, we may assume that the conical set is
${\mathcal B}:=\{\infty, 0, 1, \lambda_{1},\ldots,\lambda_{r-3}\} \subset \widehat{\mathbb C}$. Let $M_{\mathcal B}$ be the ${\rm PSL}_{2}({\mathbb C})$-stabilizer of ${\mathcal B}$, which is a finite group (the finite subgroups of ${\rm PSL}_{2}({\mathbb C})$ are either trivial, cyclic groups, dihedral groups or one of the Platonic symmetry groups ${\mathcal A}_{4}$, ${\mathcal A}_{5}$ or ${\mathfrak S}_{4}$). If ${\rm Aut}_{G}(S)$ denotes the normalizer of $G$  in ${\rm Aut}(S)$, then  there is a short exact sequence 
$$1 \to G \to {\rm Aut}_{G}(S) \stackrel{\theta}{\to} L \to 1,$$ where $L$ is a suitable subgroup of $M_{\mathcal B}$.

If $G$ is a $p$-group of type $r \geq 3$, then the Riemann-Hurwitz formula asserts that either $p \leq g+1$ or $p=2g+1$.
In \cite{H1}, it was seen that, for $p \geq 5r-7$, the group $G$ is the unique $p$-Sylow subgroup of ${\rm Aut}(S)$ (in particular, the unique $p$-subgroup of type $r$). Also, if $G \cong {\mathbb Z}_{p}$, then Castelnuovo-Severi's inequality \cite{Severi} also asserts uniqueness for  $p<r/2$ and, for $r=4$, the uniqueness holds for $p \geq 7$ \cite{HL}. 
The uniqueness property of $G$ asserts the existence of 
a natural holomorphic embedding of the modul space ${\mathfrak M}_{0,r}$, of the $r$-marked sphere, into the moduli space ${\mathfrak M}_{g}$.

Let us now restrict to the abelian case $G \cong {\mathbb Z}_{p}^{l}$, where $l \geq 1$, of type $r \geq 3$ (so, necessarily, $l \leq r-1$).
By the Riemann-Hurwitz formula, 
$2 \leq g=1+p^{l-1}((p-1)(r-2)-2)/2,$
 in particular, $(p-1)(r-2)>2$.

\subsection*{1}
If $l=1$, then $G \cong {\mathbb Z}_{p}$. In \cite{Gabino}, Gonz\'alez-Diez observed that $G$ is unique up to conjugation in ${\rm Aut}(S)$. As a consequence of the above mentioned results in \cite{Severi,H1,HL}, the uniqueness for $G$ of type $r \geq 3$ holds for either: (i) $r=3$ and $p \geq 11$,  (ii) $r=4$ and $p \geq 7$ 
and (iii) $r \geq 5$ and either $ p \geq 5r-7$ or  $p<r/2$. An equation for $S$ is well known to be of the form \cite{Wootton}
$$(*)\quad y^{p}=x(x-1)^{m}\prod_{j=1}^{r-3}(x-\lambda_{j})^{m_{j}},$$ 
where $m, m_{j} \in \{1,\ldots,p-1\}$ are such that $1+m+m_{1}+\cdots+m_{r-3} \nequiv 0 \mod p$, and a generator for $G$ is given by $a(x,y)=(x,e^{2 \pi i/p} y)$. Under the above uniqueness property of $G$, it is not difficult to search for algebraic equations for the automorphisms of $S$: 
For each $T \in M_{\mathcal B}$ one looks for some rational map $T_{1}(x,y)=v$ such that, for $u=T(x)$, it holds that $(u,v)$ satisfies the above algebraic equation $(*)$. If so, then 
$\widehat{T}(x,y)=(T(x),T_{1}(x,y))$ provides an automorphism (and all automorphisms of $S$ are so obtained).
In the recent paper \cite{RR} (see also, \cite{IJR}), the authors have considered the case $r=4$ and $p \geq 7$, where they provide an algebraic description, a realization of the groups of conformal automorphisms, the computation of their fields of moduli and they also describe isogenous decompositions of their Jacobian varieties.

\subsection*{2}
If $l=r-1$, then $G \cong {\mathbb Z}_{p}^{r-1}$. Set $\Lambda=(\lambda_{1},\ldots,\lambda_{r-3})$. In \cite{GHL} (see Section \ref{Sec:GFC}), it was observed that $S$ can be described by a an irreducible and smooth projective algebraic curve  $C_{\Lambda} \subset {\mathbb P}^{n}$, this being a certain fiber product of $r-2$ classical Fermat curves of degree $p$, 
and, in such a model, the group $G$ coincides with $H=\langle a_{1},\ldots,a_{r-1}\rangle$, where 
$$a_{j}([x_{1}:\ldots:x_{r}])=[x_{1}:\ldots:x_{j-1}:e^{2 \pi i/p} x_{j}: x_{j+1},\ldots:x_{r}], \; j=1,\ldots, r-1.$$

The group $H$ (respectively, the curve $C_{\Lambda}$) is called a generalized Fermat group (respectively, a generalized Fermat curve) of type $(p,r-1)$. 
If $r=3$ (so $p \geq 5$ and ${\mathcal B}=\{\infty,0,1\}$), then $C_{\Lambda}$ is the classical Fermat curve of degree $p$, $F_{p}:=\{x_{1}^{p}+x_{2}^{p}+x_{3}^{p}=0\} \subset {\mathbb P}^{2}$, whose group of automorphisms is isomorphic to ${\mathbb Z}_{p}^{2} \rtimes {\mathfrak S}_{3}$ (the ${\mathbb Z}_{p}^{2}$ factor corresponds to $H$ and the symmetric group factor acts by permutation of the coordinates). 
In \cite{HKLP} (see also \cite{FGHL}), it was proved that, for $(p-1)(r-2)>2$, the generalized Fermat group $H$ is unique in ${\rm Aut}(C_{\Lambda})$ (this result holds for $p$ not necessarily a prime integer) and that ${\rm Aut}(C_{\Lambda})$ consists, in such a model, by linear projective transformations. In \cite{GHL}, it was observed that $C_{\Lambda}$ is uniformized by the derived subgroup of a Fuchsian group uniformizing the orbifold $C_{\Lambda}/H$. This last fact, together the uniqueness of $H$, asserts that $L=M_{\mathcal B}$ and $|{\rm Aut}(C_{\Lambda})|=p^{r}|M_{\mathcal B}|$. 
In \cite{CHQ}, it was described an explicit isogenous decomposition of the jacobian variety $JC_{\Lambda}$, in terms of the above cone points. It says that $JC_{\Lambda} \cong_{isog} \prod_{Z} JS_{Z}$, where $Z$ runs over all subgroups of $H$ isomorphic to ${\mathbb Z}_{p}^{r-1}$ and $S_{Z}$ is the underlying Riemann surface structure of the orbifold $C_{\Lambda}/Z$ (whose equations are explicitly given).

\subsubsection*{\bf 3}
If $2 \leq l \leq r-2$, then there is some subgroup $K \cong {\mathbb Z}_{p}^{r-1-l}$ of $H \cong {\mathbb Z}_{p}^{r-1}$ (as described in 2. above), acting freely on $C_{\Lambda}$, such that $S=C_{\Lambda}/K$ and $G=H/K$ (see Theorem \ref{teo1}).  In particular, if we assume $p \geq 5r-7$, then ${\rm Aut}(S)=\widehat{G}/K$, where $\widehat{G}$ is the normalizer of $K$ in ${\rm Aut}(C_{\Lambda})$. Once the group $K$ is given, one may apply classical GIT theory to obtain (affine) algebraic equations for $S$. An isogenous decompositions for $JS$ should be obtained from the previously decomposition obtained for $JC_{\Lambda}$.

\medskip

The aim of this paper is to make a precise description of the above for $r=4$. In this case, $l \in \{1,2,3\}$, $p \geq 3$, $\Lambda=\lambda \in {\mathbb C} \setminus \{0,1\}$, ${\mathcal B}=\{\infty,0,1,\lambda\}$, $S=C_{\Lambda}/K$ and $G=H/K \cong {\mathbb Z}_{p}^{l}$, where $K \cong {\mathbb Z}_{p}^{3-l}$ is a subgroup of $H$ acting freely on $C_{\Lambda}$. 
The case $l=3$ corresponds to $S=C_{\Lambda}$, the group $G=H$ is unique; this case is discussed in Section \ref{Sec:(p,3)}. The well known situation $l=1$ (where the uniqueness of $G$ holds for $p \geq 7$) is recalled in Section \ref{Sec:l=1} as a matter of completeness. The case $l=2$ (where the uniqueness of $G$ holds if $p \geq 13$) is discussed in Section \ref{Sec:caso:l=2}.
 We provide an explicit description of those subgroups $K \cong {\mathbb Z}_{p}$, the corresponding Riemann surfaces $S=C_{\Lambda}/K$ and their groups of automorphisms (Theorem \ref{teol=2}) and 
an isogenous decomposition of the jacobian variety $JS$ (Theorem \ref{isogeno}). This decomposition asserts that $JS \cong_{isog} \prod_{Z} JS_{Z}$,
where $Z$ runs over all subgroups of $G \cong {\mathbb Z}_{p}^{2}$ isomorphic to ${\mathbb Z}_{p}$ and $S_{Z}$ is the underlying Riemann surface structure of the orbifold $S/Z$ (whose equations are explicitly given).

\section{Some generalities on generalized Fermat curves of type $(p,n)$}\label{Sec:GFC}
Let us fix integers $n,p \geq 2$.
A closed Riemann surface $S$ is called a generalized Fermat Riemann surface of type $(p,n)$ if it admits a group $G \cong {\mathbb Z}_{p}^{n}$ of conformal automorphisms such that the quotient orbifold $S/G$ has genus zero and axactly $n+1$ cone points (each one of cone order $p$); we also say that $G$ (respectively, $(S,G)$) is a generalized Fermat group (respectively, generalized Fermat pair) of type $(p,n)$. In this section, we recall some of the properies of such objects which can be found, for instance, in  \cite{GHL,HKLP}. As the case $n=2$ corresponds to classical Fermat curves of degree $p$, we will assume $n \geq 3$.

\subsection{Generalized Fermat curves of type $(p,n)$, where $n \geq 3$, and its generalized Fermat group}
Let us now assume $n \geq 3$. Set  $\omega_{p}=e^{2 \pi i/p}$ and $\Omega_{n} \subset {\mathbb C}^{n-2}$ be the domain consisting of the tuples $\Lambda=(\lambda_{1},\ldots,\lambda_{n-2})$ such that $\lambda_{j} \neq 0,1$ and $\lambda_{i} \neq \lambda_{j}$ for $i \neq j$.
Each tuple $\Lambda=(\lambda_{1},\ldots,\lambda_{n-2}) \in \Omega_{n}$ defines the algebraic curve 
$$C_{\Lambda}:=\left\{ \begin{array}{c}
x_{1}^{p}+x_{2}^{p}+x_{3}^{p}=0\\
\lambda_{1} x_{1}^{p}+x_{2}^{p}+x_{4}^{p}=0\\
\vdots\\
\lambda_{n-2} x_{1}^{p}+x_{2}^{p}+x_{n+1}^{p}=0
\end{array}
\right\} \subset {\mathbb P}^{n}.
$$

It can be checked that $C_{\Lambda}$ is an smooth irreducible algebraic curve, so it defines a closed Riemann surface.
Each of the following $n+1$ linear projective transformations 
$$
\begin{array}{c}
a_{j}([x_{1}:\cdots:x_{n+1}])=[x_{1}:\cdots:x_{j-1}:\omega_{p}x_{j}:x_{j+1}:\cdots:x_{n+1}],\; j=1,\ldots,n+1,
\end{array}
$$
defines conformal automorphisms of $C_{\Lambda}$ satisfying that (i) $a_{1}\cdots a_{n+1}=1$ and (ii) $H=\langle a_{1},\ldots,a_{n}\rangle \cong {\mathbb Z}_{p}^{n}$.
Moreover, the map 
$$\pi:C_{\Lambda} \to \widehat{\mathbb C}: [x_{1}:\cdots:x_{n+1}] \mapsto -\left(\frac{x_{2}}{x_{1}}\right)^{p},$$
is a regular branched covering, with deck group $H$, and whose branch value set is 
$${\mathcal B}_{\Lambda}:=\{\infty, 0, 1, \lambda_{1},\ldots, \lambda_{n-2}\}.$$ 

\begin{rema}\label{Obs1}
The only elements of $H$, different from the identity, with fixed points on $C_{\lambda}$ are the powers of $a_{1}, \ldots, a_{n+1}$ (these elements are called the standard generators of $H$).
\end{rema}

It follows that $(C_{\Lambda},H)$ is a generalized Fermat pair of type $(p,n)$. The Riemann-Hurwitz formula asserts that $C_{\Lambda}$ has genus 
$$g_{C_{\Lambda}}=1+p^{n-1}((p-1)(n-1)-2)/2.$$

In particular, $g_{C_{\Lambda}}>1$ if and only if $(p-1)(n-1)>2$.

\begin{theo}[\cite{FGHL,HKLP}]\label{unico}
Let $(S,G)$ be a generalized Fermat pair of type $(p,n)$, where $n \geq 3$. Up to a M\"obius transformation we may assume the set of cone points of $S/G$ (which is identified with the Riemann sphere) is given by the set $\{\infty, 0, 1, \lambda_{1},\ldots, \lambda_{n-2}\}$. Then
\begin{enumerate}
\item there is a biholomorphism $\phi:S \to C_{\Lambda}$, where $\Lambda=(\lambda_{1},\ldots,\lambda_{n-2}) \in \Omega_{n}$, such that $\phi G \phi^{-1}=H$, and
\item if $(p-1)(n-1)>2$, then $G$ is the unique generalized Fermat group of type $(p,n)$ of $S$. 
\end{enumerate}
\end{theo}

The above result asserts that we may restrict  to work with the generalized Fermat pairs $(C_{\Lambda},H)$ as defined above.

\subsection{Fuchsian uniformization}\label{Sec:uniform}
Let $(C_{\Lambda},H)$ be a generalized Fermat pair of type $(p,n)$, where $(p-1)(n-1)>2$, and 
let $\Gamma<{\rm PSL}_{2}({\mathbb R})$ be a Fuchsian group such that ${\mathbb H}^{2}/\Gamma=C_{\Lambda}/H$. A presentation of $\Gamma$ is given as follows
$$\Gamma=\langle \delta_{1},\ldots,\delta_{n+1}: \delta_{1}^{p}=\cdots=\delta_{n+1}^{p}=\delta_{1}\delta_{2}\cdots \delta_{n+1}=1\rangle.$$

If $\Gamma'$ denotes the derived subgroup of $\Gamma$, then $\Gamma/\Gamma' \cong {\mathbb Z}_{p}^{n}$, $\Gamma'$ is torsion free and ${\mathbb H}^{2}/\Gamma'$ is a closed Riemann surface. In particular, $({\mathbb H}^{2}/\Gamma', \Gamma/\Gamma')$ is a generalized Fermat pair of type $(p,n)$. 

\begin{theo}[\cite{GHL}]\label{derivado}
There is a biholomorphism $\phi:C_{\Lambda} \to {\mathbb H}^{2}/\Gamma'$ such that $\phi H \phi^{-1}=\Gamma/\Gamma'$.
\end{theo}

Since, under the assumption that $(p-1)(n-1)>2$, the generalized Fermat group $H$ in ${\rm Aut}(C_{\Lambda})$ is unique, there is 
a homomorphism $\rho:{\rm Aut}(C_{\Lambda}) \to M_{{\mathcal B}_{\Lambda}}$, where $M_{{\mathcal B}_{\Lambda}}$ is the ${\rm PSL}_{2}({\mathbb C})$-stabilizer of ${\mathcal B}_{\Lambda}$, whose kernel is $H$. As $\Gamma'$ is characteristic subgroup of $\Gamma$, the homomorphism $\theta$ is surjective, so there is a natural short exact sequence
$$1 \to H \to {\rm Aut}(C_{\Lambda}) \stackrel{\rho}{\to} M_{{\mathcal B}_{\Lambda}} \to 1.$$

\subsection{Equivalence of generalized Fermat curves}\label{Obsclases}
Above we have considered the set $\Omega_{n}$, which provides a parametrization of generelized Fermat curves of type $(p,n)$, where $(p-1)(n-1)>2$.
Let $\Lambda=(\lambda_{1},\ldots,\lambda_{n-2}), \Sigma=(\mu_{1},\ldots,\mu_{n-2}) \in \Omega_{n}$ and consider the associated generalized Fermat curves $C_{\Lambda}$ and $C_{\Sigma}$. Note that the linear group $H$ (seen as a subgroup of ${\rm PGL}_{n+1}({\mathbb C})$) restricts to both curves as the corresponding generalized Fermat group; so we still denoting them by the same letter $H$.
The uniqueness of $H$ asserts that if $\phi:C_{\Lambda} \to C_{\Sigma}$ is an isomorphism, then it descends to a M\"obius transformation $T_{\phi} \in {\rm PSL}_{2}({\mathbb C})$ sending the set ${\mathcal B}_{\Lambda}=\{\infty,0,1,\lambda_{1},\ldots,\lambda_{n-2}\}$ onto ${\mathcal B}_{\Sigma}=\{\infty,0,1,\mu_{1},\ldots,\mu_{n-2}\}$. Conversely, by Theorem \ref{derivado}, for every M\"obius transformation $T \in {\rm PSL}_{2}({\mathbb C})$ sending the set ${\mathcal B}_{\Lambda}$ onto ${\mathcal B}_{\Sigma}$ there is an isomorphism $\phi$ such that $T=T_{\phi}$. As a consequence is the following.

\begin{theo}[\cite{GHL}]\label{moduli}
If $\Lambda, \Sigma \in \Omega_{n}$ and $(p-1)(n-1)>2$, then the generalized Fermat curves (both of type $(p,n)$) $C_{\Lambda}$ and $C_{\Sigma}$ are isomorphic if and only if 
$\Lambda$ and $\Sigma$ belong to the same orbit under the group ${\mathbb G}_{n}$ of  automorphisms of $\Omega_{n}$ generated by the transformations 
$$U(\lambda_{1},\ldots,\lambda_{n-2})=\left(\frac{\lambda_{n-2}}{\lambda_{n-2}-1}, \frac{\lambda_{n-2}}{\lambda_{n-2}-\lambda_{1}}, \cdots, \frac{\lambda_{n-2}}{\lambda_{n-2}-\lambda_{n-3}} \right)$$
$$V(\lambda_{1},\ldots,\lambda_{n-2})=\left(\frac{1}{\lambda_{1}},\cdots,\frac{1}{\lambda_{n-2}}\right).$$
\end{theo}

\begin{rema}
Observe that, for $n \geq 4$, ${\mathbb G}_{n} \cong {\mathfrak S}_{n+1}$ and that ${\mathbb G}_{3} \cong {\mathfrak S}_{3}$.
In this way, the moduli space of generalized Fermat curves of type $(k,n)$, where $k \geq 2$ and $n \geq 3$, is provided by the geometric quotient $\Omega_{n}/{\mathbb G}_{n}$; which happens to be the moduli space of $(n+1)$ points in $\widehat{\mathbb C}$.
\end{rema}

\subsection{A universal property on generalized Fermat curves}
Let $S$ be a given closed Riemann surface admitting a group $G \cong {\mathbb Z}_{p}^{l}$, $l \geq 1$, of conformal automorphisms of type $n+1$ (so $l \leq n$), 
i.e., $S/G$ is the Riemann sphere $\widehat{\mathbb C}$ with exactly $n+1$ cone points, which (up to M\"obius transformations) might be assumed to be given by the set
$${\mathcal B}_{\Lambda}=\{\infty,0,1,\lambda_{1},\ldots,\lambda_{n-2}\},$$
where $\Lambda=(\lambda_{1},\ldots,\lambda_{n-2}) \in \Omega_{n}$. We also assume that each cone point has order $p$ (this is the situation if $p$ is assumed to be a prime integer).

If ${\rm Aut}_{G}(S)$ denotes the normalizer of $G$ in ${\rm Aut}(S)$, then we have a short exact sequence
$$1 \to G \to {\rm Aut}_{G}(S)  \stackrel{\theta}{\to} L \to 1,$$
where $L$ is a suitable subgroup of $M_{{\mathcal B}_{\Lambda}}$. As noted in the introduction, for $p$ a prime integer, the normality (in fact, the uniqueness)  of $G$ is ensured when either: (i) $p \geq 7$, $l=1$ and $n=3$, or (ii) $p \geq 5n-2$ and $n \geq 3$

\begin{theo}\label{teo1}
If  $(S,G)$ and $\Lambda$ are as as above, then the following hold.

\begin{enumerate}
\item There is a subgroup $K \cong {\mathbb Z}_{p}^{n-l}$ of $H$, acting freely on $C_{\lambda}$, and a biholomorphism $\varphi:S \to C_{\lambda}/K$ such that
$\varphi G \varphi^{-1}=H/K$. 

\item Let us assume $p$ is a prime integer. If either: (i) $p \geq 7$, $l=1$ and $n=3$, or (ii) $p \geq 5n-2$ and $n \geq 3$, then ${\rm Aut}(S)\cong \widehat{G}/K$, where $\widehat{G}<{\rm Aut}(C_{\Lambda})$ is the ${\rm Aut}(C_{\Lambda})$-stabilizer of $K$. In particular, $|{\rm Aut}(S)|=p^{l} |\theta({\rm Aut}(S))|$ and $|\theta({\rm Aut}(S))|$ is a divisor of $|M_{{\mathcal B}_{\Lambda}}|$.

\end{enumerate}
\end{theo}
\begin{proof}
As previously seen, if $\Gamma<{\rm PSL}_{2}({\mathbb R})$ is a Fuchsian group, acting on the hyperbolic plane ${\mathbb H}^{2}$, such that ${\mathbb H}^{2}/\Gamma$ is the Riemann orbifold $C_{\Lambda}/H=S/G$, then $C_{\Lambda}={\mathbb H}^{2}/\Gamma'$, where $\Gamma'$ denotes the derived subgroup of $\Gamma$. In particular, this asserts the existence of a normal subgroup $\Gamma_{S}$ of $\Gamma$ such that $\Gamma' \lhd \Gamma_{S} \lhd \Gamma$, $S={\mathbb H}^{2}/\Gamma_{S}$ and $G=\Gamma/\Gamma_{S}$. In other words, there is a subgroup $K$ of $H$, acting freely on $C_{\Lambda}$ such that $S=C_{\Lambda}/K$ and $G=H/K$. We note that $K \cong {\mathbb Z}_{p}^{n-l}$ (as we are assuming the cone points of $S/G$ to be each one of order $p$) and we have obtained part (1).
In order to obtain part (2), let us first observe that the imposed conditions on the prime integer $p$ ensures the uniquenes of $G$ in ${\rm Aut}(S)$, so we have the short exact sequence $1 \to G \to {\rm Aut}(S) \stackrel{\theta}{\to} L=\theta({\rm Aut}(S)) \to 1$, where $L<M_{{\mathcal B}_{\Lambda}}$.
If  $\widehat{G}=\rho^{-1}(L)<{\rm Aut}(C_{\lambda})$ (where $\rho:{\rm Aut}(C_{\Lambda}) \to M_{{\mathcal B}_{\Lambda}}$ is the surjective homomorphism of Section \ref{Sec:uniform}), then $K$ must be invariant under $\widehat{G}$ (as it is in fact a lifting to $C_{\Lambda}$ of ${\rm Aut}(S)$).
\end{proof}

\section{Generalized Fermat pairs of type $(p,3)$}\label{Sec:(p,3)}
In this section, we restrict to the case of generalized Fermat pairs of type $(p,3)$, where $p \geq 3$.
We set $C_{\Lambda}=C_{\lambda}$, ${\mathcal B}_{\Lambda}={\mathcal B}_{\lambda}=\{\infty,0,1,\lambda\}$, where $\lambda \in {\mathbb C} \setminus \{0,1\}$. In this case, 
$M_{\Lambda}=M_{\lambda}$ contains the subgroup 
$$M_{2,\lambda}:=\langle A(z)=\lambda/z, B(z)=(z-\lambda)/(z-1)\rangle \cong {\mathbb Z}_{2}^{2}.$$

Theorem \ref{moduli}, in our situation, reads as follows.

\begin{coro}[$n=3$]\label{coromoduli}
Let $\lambda \neq \mu \in \Omega_{1}={\mathbb C} \setminus \{0,1\}$. Then $C_{\lambda}$ and $C_{\mu}$ are isomorphic if and only if 
$j(\lambda)=j(\mu)$, where $j(x)=(1-x+x^{2})^{2}/x^{2}(x-1)^{2}$ is the classical elliptic modular function, that is, if and only if 
$\lambda$ and $\mu$ are in the same ${\mathbb G}$-orbit, where
$${\mathbb G}=\langle T(z)=1/z, R(z)=1-z\rangle \cong {\mathfrak S}_{3}.$$
Moreover, the ${\mathbb G}$-stabilizer of $\lambda$ is isomorphic to ${\rm Aut}(C_{\lambda})/H \cong M_{\lambda}$.
\end{coro}

\begin{rema}[$n=3$]
As a consequence of the above, for the case $n=3$, we observe that 
the collections $\{-1, 1/2, 2\}$ and $\left\{\left(1+i\sqrt{3}\;\right)/2, (1-i\sqrt{3})/2\right\}$ form two isomorphism classes. If $\lambda \in {\mathbb C} \setminus \{0,1\}$  is different from the previous, then its orbit consists of exactly $6$ points. 
\end{rema}

\subsection{The group of automorphisms of $C_{\lambda}$}
Let us satrt by noting the following facts.
\begin{enumerate}
\item If $\lambda \notin \left\{-1,1/2,2, \left(1 \pm i \sqrt{3}\;\right)/2\right\}$, then $M_{\lambda}=M_{2,\lambda}$. 
\item If $\lambda \in \{-1,1/2,2\}$, then $M_{\lambda} \cong D_{4}$ (the dihedral group of order $8$) as it contains an extra involution $C$: (i) $C(z)=1/z$ for $\lambda=-1$, (ii) $C(z)=z/(2z-1)$ for $\lambda=1/2$ and (iii) $C(z)=z/(z-1)$ for $\lambda=2$.
\item If $\lambda \in \left\{\left(1 \pm i \sqrt{3}\;\right)/2\}\right\}$, then $M_{\lambda} \cong {\mathcal A}_{4}$ (the alternating group of order $12$) as it contains the order three rotation $D(z)=1+\left(1 + i \sqrt{3}\;\right)z/2$.
\end{enumerate}

Let $\sqrt[p]{2}$ be the real $p$-root of $2$ and fix $p$-roots $\sqrt[p]{\lambda}$ and $\sqrt[p]{\lambda-1}$.
It is not so difficult to check that the following linear projective transformations keep $C_{\lambda}$ invariant (so they define conformal automorphisms):

\begin{equation}
\left\{\begin{array}{l}
\widehat{\alpha}([x_{1}:x_{2}:x_{3}:x_{4}])=[x_{2}:\sqrt[p]{\lambda}\; x_{1}: x_{4}: \sqrt[p]{\lambda}\; x_{3}]\\
\widehat{\beta}([x_{1}:x_{2}:x_{3}:x_{4}])=[-x_{3}: x_{4}:\sqrt[p]{\lambda-1}\; x_{1}:-\sqrt[p]{\lambda-1}\; x_{2}]\\
\widehat{\gamma}([x_{1}:x_{2}:x_{3}:x_{4}])=[x_{4}:\sqrt[p]{2}\; x_{1}:x_{2}:-\sqrt[p]{2}\;x_{3}], (\lambda=2)\\
\widehat{\gamma}([x_{1}:x_{2}:x_{3}:x_{4}])=[x_{3}:x_{4}:\sqrt[p]{2}\;x_{2}:-\sqrt[p]{2}\;x_{1}], (\lambda=-1)\\
\widehat{\gamma}([x_{1}:x_{2}:x_{3}:x_{4}])=[\sqrt[p]{2}\;x_{3}:x_{1}:-\sqrt[p]{2}\;x_{4}:x_{2}], (\lambda=1/2)\\
\widehat{\delta}([x_{1}:x_{2}:x_{3}:x_{4}])=[\sqrt[p]{\lambda}\; x_{1}:-x_{4}:x_{2}:x_{3}], (\lambda=\left(1+i\sqrt{3}\;\right)/2)\\
\widehat{\delta}([x_{1}:x_{2}:x_{3}:x_{4}])=[\sqrt[p]{\lambda}\; x_{1}:x_{4}:x_{2}:-x_{3}], (\lambda=\left(1-i\sqrt{3}\;\right)/2).
\end{array}
\right.
\end{equation}

Let us note that 
(i) $\langle \widehat{\alpha}, \widehat{\beta}\rangle\cong {\mathbb Z}_{2}^{2}$,  
(ii) $\widehat{\gamma}^{4}=1$ and (iii) $\widehat{\delta}^{3}=1$.  Also, $\rho(\langle \widehat{\alpha}, \widehat{\beta}\rangle)=M_{2,\lambda}$, so $\rho:\langle \widehat{\alpha}, \widehat{\beta}\rangle \to M_{2,\lambda}$ is an isomorphism.

\begin{theo}[\cite{FGHL}]\label{automorfismos}
The descrition of ${\rm Aut}(C_{\lambda})$ is as follows.

\begin{itemize}
\item[(i)] If $\lambda \notin \{-1,1/2,2,(1\pm i\sqrt{3})/2\}$, then ${\rm Aut}(C_{\lambda})/H \cong {\mathbb Z}_{2}^{2}$, ${\rm Aut}(C_{\lambda})= H \rtimes \langle \widehat{\alpha}, \widehat{\beta}\rangle$ and the signature of $C_{\lambda}/{\rm Aut}(C_{\lambda})$ is $(0;2,2,2,p)$;

\item[(ii)] if $\lambda \in \{-1,1/2,2\}$, then ${\rm Aut}(C_{\lambda})/H \cong D_{4}$, ${\rm Aut}(C_{\lambda})= H \rtimes \langle\widehat{\alpha}, \widehat{\beta}, \widehat{\gamma} \rangle$ and the signature of $C_{\lambda}/{\rm Aut}(C_{\lambda})$ is $(0;2,4,2p)$;

\item[(iii)] if $\lambda \in \{\left(1+i\sqrt{3}\;\right)/2,\left(1-i\sqrt{3}\;\right)/2 \}$, then ${\rm Aut}(C_{\lambda})/H \cong {\mathcal A}_{4}$, ${\rm Aut}(C_{\lambda})= H \rtimes \langle\widehat{\alpha}, \widehat{\beta}, \widehat{\delta} \rangle$ and the signature of $C_{\lambda}/{\rm Aut}(C_{\lambda})$ is $(0;2,3,3p)$.
\end{itemize}
\end{theo}

\subsection{On the field of moduli of $C_{\lambda}$}
The field of moduli ${\mathcal M}_{\lambda}$ of $C_{\lambda}$ is defined as the fixed field of the group $G_{\lambda}=\{\rho \in {\rm Gal}({\mathbb C}): C_{\rho(\lambda)} \cong C_{\lambda}\}$, where ${\rm Gal}({\mathbb C})$ denotes the group of field automorphisms of the complex numbers and the symbol $\cong$ denotes isomorphism of curves. ${\mathcal M}_{\lambda}$ is the field of definition of the isomorphic class of $C_{\lambda}$ in moduli space and it coincides with the intersection of all the fields of definition of $C_{\lambda}$ \cite{Koizumi}. By Corollary \ref{coromoduli}, ${\mathcal M}_{\lambda}={\mathbb Q}(j(\lambda))$.
In case (i), of Theorem \ref{automorfismos}, these have odd signature \cite{AQ}, so they are definable over their field of moduli ${\mathbb Q}(j(\lambda))$.
The cases (ii) and (iii)  are quasiplatonic curves, so they are also definable over their corresponding field of moduli \cite{Wolfart} (in these cases this field being  ${\mathbb Q}$). Note that in case (ii) the curves are already defined over its field of moduli ${\mathbb Q}$. The curve $C_{\left(1\pm i\sqrt{3}\;\right)/2}$ is defined over a quadratic extension of its field of moduli.

\subsection{A combinatorial point of view} \label{combinatoria}
Let $H$ be the generalized Fermat group of type $(p,3)$ of $C_{\lambda}$.
As $H$ is a normal subgroup, we may look at the conjugation action of ${\rm Aut}(C_{\lambda})$ on $H$. This action can be seen in a combinatorial simple form \cite{GHL} by looking at the short exact sequence $1 \to H \to {\rm Aut}(C_{\lambda}) \stackrel{\rho}{\to} M_{\lambda} \to 1$. We proceed to describe it below.

If $U \in {\rm Aut}(C_{\lambda})$, then $\rho(U) \in M_{\lambda}$ induces a permutation on the points $p_{1}=\infty$, $p_{2}=0$, $p_{3}=1$ and $p_{4}=\lambda$, so a permutation $\sigma_{U} \in {\mathfrak S}_{4}$. It can be seen that, for each $j \in \{1,2,3,4\}$, $U  a_{j}  U^{-1}=a_{\sigma_{U}(j)}$.
This provides a permutation representation 
$$\Theta:{\rm Aut}(C_{\lambda}) \to {\mathfrak S}_{4}: U \mapsto \Theta(U)=\sigma_{U},$$
whose kernel is $H$. In particular, $\Theta({\rm Aut}(C_{\lambda})) \cong {\rm Aut}(C_{\lambda})/H \cong M_{\lambda}$. 

For the above given automorphisms $\widehat{\alpha}$, $\widehat{\gamma}$ and $\widehat{\delta}$ we have:
$$
\begin{array}{c}
\Theta(\widehat{\alpha})=(1,2)(3,4), \; \theta(\widehat{\beta})=(1,3)(2,4),\\
\Theta(\widehat{\gamma})=(1,4,2,3), \; (\lambda=-1),\\
\Theta(\widehat{\gamma})=(2,4,3,1), \; (\lambda=1/2),\\
\Theta(\widehat{\gamma})=(1,2,3,4), \; (\lambda=2),\\
\Theta(\widehat{\delta})=(2,3,4), \; \left(\lambda \in \left\{\left(1+i\sqrt{3}\;\right)/2, \left(1-i\sqrt{3}\;\right)/2 \right\}\right).
\end{array}
$$

Note that $\Theta(\widehat{\gamma})^{2} \in \langle \Theta(\widehat{\alpha}),\Theta(\widehat{\beta})\rangle \cong {\mathbb Z}_{2}^{2}$. 

The above explicit representations (together Theorem \ref{automorfismos}) permit to list the only three (up to conjugation in ${\mathfrak S}_{4}$) possible images $\Theta({\rm Aut}(C_{\lambda}))$:
$${\mathcal U}_{2}=\langle (1,2)(3,4), (1,3)(2,4) \rangle \cong {\mathbb Z}_{2}^{2}, \quad \lambda \notin \{-1,1/2,2,(1\pm i\sqrt{3})/2\}.$$
$${\mathcal U}_{8}=\langle (1,2)(3,4), (1,2,3,4) \rangle \cong D_{4} \quad (\lambda =2).$$
$${\mathcal U}_{12}=\langle (1,2)(3,4), (2,3,4) \rangle \cong {\mathcal A}_{4} \quad \left(\lambda =\left(1 \pm i\sqrt{3}\;\right)/2 \right).$$

We also consider the following subgroups of ${\mathfrak S}_{4}$, which are obtained by restricting $\Theta$ to subgroups of ${\rm Aut}(C_{\lambda})$ containing $H$ (again, up to conjugation):
$${\mathcal U}_{0}=\{()\}, \; {\mathcal U}_{1}=\langle (1,2)(3,4) \rangle \cong {\mathbb Z}_{2},$$
$${\mathcal U}_{3}=\langle (2,3,4) \rangle \cong {\mathbb Z}_{3} \quad \left(\lambda =\left(1 \pm i\sqrt{3}\;\right)/2 \right), \;
{\mathcal U}_{4}=\langle (1,2,3,4) \rangle \cong {\mathbb Z}_{4} \quad (\lambda =2).$$

\begin{lemm}\label{lema1}
Let us assume $p \geq 3$ is a prime integer.
If $\Theta({\rm Aut}(C_{\lambda}))={\mathcal U}_{12}$, then every non-trivial subgroup of $H$ which is ${\rm Aut}(C_{\lambda})$-invariant under conjugation contains 
elements acting with fixed points on $C_{\lambda}$.
\end{lemm}
\begin{proof}
Let us recall (Remark \ref{Obs1}) that the non-trivial elements acting with fixed points on $C_{\lambda}$ are exactly the powers of $a_{1}$, $a_{2}$, $a_{3}$ and $a_{4}$.
The condition $\Theta({\rm Aut}(C_{\lambda}))={\mathcal U}_{12}$ asserts that $\widehat{\alpha}, \widehat{\delta} \in {\rm Aut}(C_{\lambda})$. 
Let $K$ be a non-trivial subgroup of $H$, being ${\rm Aut}(C_{\lambda})$-invariant under conjugation.
If $a=a_{1}^{r_{1}}a_{2}^{r_{2}}a_{3}^{r_{3}}a_{4}^{r_{4}} \in K$ (where $0 \leq r_{j} \leq p-1\}$), different from the identity, 
then $b=\widehat{\alpha} a  \widehat{\alpha}^{-1}=a_{1}^{r_{2}}a_{2}^{r_{1}}a_{3}^{r_{4}}a_{4}^{r_{3}}, c=\widehat{\delta} a  \widehat{\delta}^{-1}=a_{1}^{r_{1}}a_{2}^{r_{4}}a_{3}^{r_{2}}a_{4}^{r_{3}} \in K$.
As $a_{1}a_{2}a_{3}a_{4}=1$, we may assume $r_{4}=0$. Up to a power, we also may assume (up to a permutation of the indices) that $r_{1}=1$, i.e., $a=a_{1}a_{2}^{r_{2}}a_{3}^{r_{3}}$, $b=a_{1}^{r_{2}-r_{3}}a_{2}^{1-r_{3}}a_{3}^{-r_{3}}$ and $c=a_{1}^{1-r_{3}}a_{2}^{-r_{3}}a_{3}^{r_{2}-r_{3}}$. As $ab=a_{1}^{1+r_{2}-r_{3}}a_{2}^{1+r_{2}-r_{3}} \in K$, then we may also assume $r_{3}=0$.
Now, we may assume $a=a_{1}a_{2}^{r}, b=a_{1}^{r}a_{2}, c=a_{1}a_{3}^{r} \in K$, where $0 \leq r \leq p-1$. If $r=0$, then $a=a_{1} \in K$ and we are done. Assume $r\neq 0$.
As $\widehat{\delta} a_{1}a_{3}^{r}  \widehat{\delta}^{-1}=a_{1}a_{4}^{r} \in K$, then 
$a_{1}^{3-r}=(a_{1}a_{2}^{r})(a_{1}a_{3}^{r})(a_{1}a_{4}^{r}) \in K$. So, if $r \neq 3$, we are done. Let us assume $r=3$ (so $p \geq 5$). In this case, $ab^{-3}=a_{1}^{-8} \in K$ and we are done.
\end{proof}

\begin{rema}
There is a natural action of ${\mathfrak S}_{4}$ as a group of automorphisms of $H$, defined by the rule that $\rho \in {\mathfrak S}_{4}$ sends each $a_{j}$ to $a_{\rho(j)}$. As 
$\Theta({\rm Aut}(C_{\lambda}))< {\mathfrak S}_{4}$ we may consider the restriction of such action to it. It can be seen that this is the same 
conjugation action of ${\rm Aut}(C_{\lambda})$ on $H$ as described above.
The action of ${\mathfrak S}_{4}$ on $H$ induces a natural faithful permutation action of the symmetric group ${\mathfrak S}_{4}$ on the collection of subgroups $K \cong {\mathbb Z}_{p}^{l}$, for $l \in \{1,2\}$, acting freely on $C_{\lambda}$.
\end{rema}

\begin{example}
Let us consider the transformation $a_{1}^{k_{1}}a_{2}^{k_{2}}a_{3}^{k_{3}}a_{4}^{k_{4}} \in H$, where $0 \leq k_{i} \leq p-1$. In matrix form, this is given by
$$\left[\begin{array}{cccc}
\omega_{p}^{k_{1}} & 0 & 0 & 0\\
0 & \omega_{p}^{k_{2}} & 0 & 0 \\
0 & 0 & \omega_{p}^{k_{3}} & 0 \\
0 & 0 & 0 & \omega_{p}^{k_{4}}
\end{array}
\right] \in {\rm PGL}_{3}({\mathbb C})
$$
The action of a permutation $\rho \in {\mathfrak S}_{4}$ is just the action of permutation of the elements of the diagonal. For instance, 
the action of $\Theta(\widehat{\alpha})$ on it is given by
$$\left[\begin{array}{cccc}
\omega_{p}^{k_{2}} & 0 & 0 & 0\\
0 & \omega_{p}^{k_{1}} & 0 & 0 \\
0 & 0 & \omega_{p}^{k_{4}} & 0 \\
0 & 0 & 0 & \omega_{p}^{k_{3}}
\end{array}
\right]
$$
\end{example}

\subsection{Description of subgroups of $H$ acting freely on $C_{\lambda}$}
Next, we describe those subgroups $K \cong {\mathbb Z}_{p}^{m}$, where $m \in \{1,2\}$ and $p \geq 3$ is a prime integer, of $H$ acting freely on $C_{\lambda}$. 

\subsubsection{Case $K \cong {\mathbb Z}_{p}$}\label{Sec:r=1}
As noted in Remark \ref{Obs1},  the only elements of $H$ having fixed points on $C_{\lambda}$ are the powers of $a_{1}, a_{2}, a_{3}$ and $a_{4}$. In this way, a non-trivial element $a_{1}^{k_{1}}a_{2}^{k_{2}}a_{3}^{k_{3}} \in H$, where $k_{j} \in \{0,1,\ldots,p-1\}$, will have no fixed points on $C_{\lambda}$  if it is not a power of some $a_{j}$, that is: (i) no two of the exponents $k_{i}$'s are zero and (ii) if all the exponenst are different from zero, then they are not equal.
As a consequence, the order $p$ cyclic subgroups of $H$ acting freely are the following $p^{2}+p-3$ ones:
$$\langle a_{2}a_{3}^{k} \rangle, \; k=1,\ldots, p-1,$$
$$\langle a_{1}a_{2}^{r}a_{3}^{s} \rangle, \; (r,s) \in \{0,1,\ldots,p-1\} \setminus \{(0,0), (1,1)\}.$$

\subsubsection{Case $K \cong {\mathbb Z}_{p}^{2}$}\label{Sec:r=2}
Every subgroup of $H$ isomorphic to ${\mathbb Z}_{p}^{2}$ is given as the kernel of a surjective homomorphism $\eta:H \to {\mathbb Z}_{p}=\{0,1,\ldots,p-1\}$. Any two of these homomorphisms produces the same kernel when one is obtained from the other by post-composition by an automorphism of ${\mathbb Z}_{p}$. The condition for the kernel to act freely on $C_{\lambda}$ is equivalent to have that $\eta(a_{j}) \neq 0$, for every $j=1,2,3,4$. 
Each surjective homomorphism $\eta$, for which $\eta(a_{j}) \neq 0$, for every $j=1,2,3,4$, induces a tuple $I=(I_{1},\ldots,I_{p-1})$, where  $I_{k} \subset \{1,2,3,4\}$ (it might be empty), 
$I_{k_{1}} \cap I_{k_{2}}=\emptyset$ (for $k_{1} \neq k_{2}$), $I_{1} \cup \cdots \cup I_{p-1}=\{1,2,3,4\}$ and $\sum_{k=1}^{p-1} k \#(I_{k}) \equiv 0 \mod p$. This tuple is defined by the rule that $j \in I_{k}$ if and only if $\eta(a_{j})=k$. Conversely, each tuple as above is induced by a unique such a surjective homomorphism.
Set ${\mathcal F}_{p}$ be the collection of such tuples $I=(I_{1},\ldots,I_{p-1})$.
The group ${\rm Aut}({\mathbb Z}_{p}) \cong {\mathbb Z}_{p-1}$ acts by permuting these tuples, so the above permits (with some care) to obtain that the number of equivalent classes (so the number of the subgroups of $H$ we are looking for) is $p^{2}-p-1$.

Let $I=(I_{1},\ldots,I_{p-1}) \in {\mathcal F}_{p}$. As said before, this tuple determines a unique surjective homomorphism $\eta$ whose kernel $K_{I}$ is a subgroup of $H$ acting freely on $C_{\lambda}$ and isomorphic to ${\mathbb Z}_{p}^{2}$. If $i,j \in I_{k}$, then $a_{i}a_{j}^{-1} \in K_{I}$. More generaly, assume $i \in I_{k_{1}}$ and $j \in I_{k_{2}}$, where $k_{1}<k_{2}$. As $k_{1}$ generates ${\mathbb Z}_{p}$, there is some $d \in \{1,\ldots,p-1\}$ such that $k_{2} \equiv dk_{1} \mod p$. It follows that $\eta(a_{i}^{-d}a_{j})=0$, that is, $a_{i}^{-d}a_{j} \in K_{I}$. In this way, knowing the tuple $I$ one can easily produce a set of generators for $K_{I}$. This observation permits to see that there are exactly $(p^{2}-p-1)$ different subgroups of $H$ isomorphic to ${\mathbb Z}_{p}^{2}$ and acting freely on $C_{\lambda}$

\section{Riemann surfaces with ${\mathbb Z}_{p}^{l}$ conformal actions of type $4$}

\subsection{}
Let $S$ be a given closed Riemann surface admitting a group $G \cong {\mathbb Z}_{p}^{l}$, $l  \in \{1,2\}$ and $p \geq 3$ a prime integer, of conformal automorphisms of type $4$, i.e., $S/G$ is the Riemann sphere $\widehat{\mathbb C}$ with exactly four cone points, which can be assumed to be $\infty, 0,1, \lambda$. 
As before, $M_{\lambda}$ denotes the ${\rm PSL}_{2}({\mathbb C})$-stabilizer of the set $\{\infty,0,1,\lambda\}$ (and its corresponding subgroup $M_{2,\lambda} \cong {\mathbb Z}_{2}^{2}$). By Remark \ref{Obsclases}, we will assume that either: (1) $\lambda \notin \left\{-1,1/2,2, \left(1 \pm i\sqrt{3}\;\right)/2 \right\}$ or  (2) $\lambda=2$ or (3) $\lambda=\left(1+i\sqrt{3}\;\right)/2$. 

As a consequence of Theorem \ref{teo1}, 
there is a subgroup $K \cong {\mathbb Z}_{p}^{3-l}$ of $H$, acting freely on $C_{\lambda}$, and there is a biholomorphism $\varphi:S \to C_{\lambda}/K$ such that
$\varphi G \varphi^{-1}=H/K$. So, from now on, we assume $S=C_{\lambda}/K$ and $G=H/K$.

The conjugation action of ${\rm Aut}(C_{\lambda})$ on $H$ induces a conjugation action on the collection of subgroups of $H$ acting freely on $C_{\lambda}$ and which are isomorphic to ${\mathbb Z}_{p}^{3-l}$. Let  $\widehat{G}<{\rm Aut}(C_{\lambda})$ be the stabilizer of $K$ by such an action (i.e., the normalizer of $K$ in ${\rm Aut}(C_{\lambda})$). If 
 ${\rm Aut}_{G}(S)$ is the normalizer of $G$ in ${\rm Aut}(S)$, then 
${\rm Aut}_{G}(S)=\widehat{G}/K$. We have a short exact sequence
$$1 \to G \to {\rm Aut}_{G}(S)  \stackrel{\theta}{\to} L \to 1,$$
where $L$ is a suitable subgroup of $M_{{\mathcal B}_{\Lambda}}$. Note that $\rho^{-1}(L)=\widehat{G}$.
Theorem \ref{automorfismos} may be used to obtain explicit descriptions of the groups ${\rm Aut}_{G}(S)$. 

\begin{rema}
If we also assume (i) $p \geq 7$ for $l=1$ and (ii) $p \geq 13$ for $l=2$, then ${\rm Aut}(S)={\rm Aut}_{G}(S)=\widehat{G}/K$. 
In particular, $|{\rm Aut}(S)|=p^{l} |\theta({\rm Aut}(S))|$ and $|\theta({\rm Aut}(S))|$ is a divisor of $|M_{\lambda}|$.
\end{rema}

\subsection{Case $l=1$: ${\mathbb Z}_{p}$ actions of type $4$}\label{Sec:l=1}
In this case, $S$ has genus $p-1$ and $K \cong {\mathbb Z}_{p}^{2}$ is one of the $p^{2}-p-1$ possible acting freely subgroups as described in Section \ref{Sec:r=2}.

\begin{rema}[Algebraic models]\label{Obs4}
As already noted in the introduction, $S$ can be described by an algebraic curve of the form
$$y^{p}=x(x-1)^{m}(x-\lambda)^{n},$$
where $m,n \in \{1,\ldots, p-1\}$ are such that $1+m+n \nequiv 0 \mod p$; in such a model, $G=\langle (x,y) \mapsto (x,\omega_{p} y)\rangle$, where $\omega_{p}=e^{2 \pi i/p}$. Moreover, it was explained how to explicitly construct algebraic models of the automorphisms of $S$ belonging to ${\rm Aut}_{G}(S)$, so for $p \geq 7$, all of its automorphisms.
\end{rema}

The permutation action $\Theta:{\rm Aut}(C_{\lambda}) \to {\mathfrak S}_{4}$, described in Section \ref{combinatoria}, codifies in a combinatorial way the congugacy action of ${\rm Aut}(C_{\lambda})$ on $H$ and so on the collection of its subgroups and it may be used to compute $\widehat{G}$. Take two generators of $K$, say $a_{1}^{k_{1}}a_{2}^{k_{2}}a_{3}^{k_{3}}a_{4}^{k_{4}}$ and $a_{1}^{r_{1}}a_{2}^{r_{2}}a_{3}^{r_{3}}a_{4}^{r_{4}}$. If $U \in {\rm Aut}(C_{\lambda})$ and $\sigma:=\Theta(U)$, then $U \in \widehat{G}$ if and only if  $a_{\sigma(1)}^{k_{1}}a_{\sigma(2)}^{k_{2}}a_{\sigma(3)}^{k_{3}}a_{\sigma(4)}^{k_{4}} \in K$ and $a_{\sigma(1)}^{r_{1}}a_{\sigma(2)}^{r_{2}}a_{\sigma(3)}^{r_{3}}a_{\sigma(4)}^{r_{4}} \in K$.

In Section \ref{combinatoria}, we have described the possibilities (up to conjugation) for $\Theta(F)<{\mathfrak S}_{4}$, for $H \leq F \leq {\rm Aut}(C_{\lambda})$; these were denoted as ${\mathcal U}_{j}$, for $j=0,1,2,3,4,8,12$. If $F=\widehat{G}$ (the normalizer of $K$), then $F/K={\rm Aut}_{G}(S)$.
By Lemma \ref{lema1}, ${\mathcal U}_{12} \cong {\mathcal A}_{4}$ is impossible (as we are asking to keep invariant $K$ and it has no non-trivial elements with fixed points). 

Below, we describe the possibilities of ${\rm Aut}_{G}(S)$, in terms of the others cases ${\mathcal U}_{j}$ (it is not hard to write down the corresponding algebraic models as in Remark \ref{Obs4} \cite{RR}). Recall that, for $p \geq 7$, ${\rm Aut}_{G}(S)={\rm Aut}(S)$.

\subsubsection{\bf Case ${\mathcal U}_{0}$}\label{caso1}
This happens, for instance, for $K=\langle a_{1}a_{2}, a_{1}a_{3}^{-1}\rangle$.
In this (generic) case $\widehat{G}=H$, ${\rm Aut}_{G}(S)=G \cong {\mathbb Z}_{p}$ and  $S/{\rm Aut}_{G}(S)$ has signature $(0;p,p,p,p)$. 

\subsubsection{\bf Case ${\mathcal U}_{1}$}\label{caso2}
In this case, $\widehat{G} \cong H \rtimes {\mathbb Z}_{2}$, where the ${\mathbb Z}_{2}$ factor is either one of the followings:
$ \langle \widehat{\alpha} \rangle$, $\langle \widehat{\beta} \rangle$, $\langle \widehat{\alpha}\widehat{\beta} \rangle$,
$\langle \widehat{\alpha}\widehat{\gamma} \rangle$ or $\langle \widehat{\beta}\widehat{\gamma} \rangle$.
In the first three cases, ${\rm Aut}_{G}(S) = G \rtimes {\mathbb Z}_{2} \cong D_{p}$ and the signature of $S/{\rm Aut}_{G}(S)$ is $(0;2,2,p,p)$. In the last two ones ${\rm Aut}_{G}(S)=G \times {\mathbb Z}_{2} \cong {\mathbb Z}_{2p}$ and the signature for $S/{\rm Aut}_{G}(S)$ is $(0;p,2p,2p)$. These families are in the closure of the above family in case \ref{caso1}.

\subsubsection{\bf Case ${\mathcal U}_{2}$}\label{caso3}
In this case, $\widehat{G} \cong H \rtimes {\mathbb Z}^{2}_{2}$, where the ${\mathbb Z}^{2}_{2}$ factor is 
$ \langle \widehat{\alpha}, \widehat{\beta} \rangle$.
In this case, ${\rm Aut}_{G}(S) = G \rtimes {\mathbb Z}^{2}_{2}$ and the signature of $S/{\rm Aut}_{G}(S)$ is $(0;2,2,2,p)$.  This family is in the closure of any of the first three families in case \ref{caso2}.

\subsubsection{\bf Case ${\mathcal U}_{3}$}\label{caso4}
In this case, $\widehat{G} \cong H \rtimes {\mathbb Z}_{3}$, where ${\mathbb Z}_{3}$ can be though to be generated by $\widehat{\delta}$ (or any of its $L$-conjugate), ${\rm Aut}_{G}(S)=G \times {\mathbb Z}_{3}$ ($ \cong {\mathbb Z}_{3p}$ for $p \geq 5$) and the signature for $S/{\rm Aut}_{G}(S)$ is $(0;3,p,3p)$. In this case, by Theorem \ref{automorfismos}, $\lambda=(1 \pm i \sqrt{2})/2$. This family is in the closure of the above family in case \ref{caso1}.

\subsubsection{\bf Case ${\mathcal U}_{4}$}\label{caso5}
In this case, $\widehat{G} \cong H \rtimes {\mathbb Z}_{4}$, where ${\mathbb Z}_{4}$ can be though to be generated by $\widehat{\gamma}$ (or any of its $L$-conjugate), ${\rm Aut}_{G}(S)=G \rtimes {\mathbb Z}_{4} \cong {\mathbb Z}_{p} \rtimes {\mathbb Z}_{4}$ and the signature for $S/{\rm Aut}_{G}(S)$ is $(0;4,4,p)$. In this case, by Theorem \ref{automorfismos}, $\lambda \in \{-1,1/2,2\}$ (so $S$ is definable over ${\mathbb Q}$). This family is in the closure of the family in case \ref{caso1} and also of the first three families in case \ref{caso2}.

\subsubsection{\bf Case ${\mathcal U}_{8}$}\label{caso6}
In this case, $\widehat{G} \cong H \rtimes D_{4}$, where $D_{4}$ can be though to be generated by $L$ and $\widehat{\gamma}$ (or any of its $L$-conjugate), ${\rm Aut}_{G}(S)=G \rtimes D_{4} \cong {\mathbb Z}_{p} \rtimes D_{4}$ and the signature for $S/{\rm Aut}_{G}(S)$ is $(0;2,4,2p)$. It is well known that, up to isomorphisms, there is only one possible such a Riemann surface $S$, this being Accola-Maclachlan curve $y^{2}=x^{2p}-1$ and ${\rm Aut}_{G}(S)={\rm Aut}(S)$. Note that this curve belongs to the clousure of the families in cases \ref{caso3} and  \ref{caso5} \cite{CI}. See also Example \ref{ejemplo2}.

\begin{rema}[Non-maximal signatures]\label{obsnomax}
Note that in each of the cases \ref{caso1}, \ref{caso2}, \ref{caso4} and \ref{caso5}, the signature of the quotient orbifold $S/{\rm Aut}_{G}(S)$ 
is non-maximal \cite{Singerman}, that is, if $\Delta$ is a Fuchsian group such that ${\mathbb H}^{2}/\Delta=S/{\rm Aut}_{G}(S)$, then there is another Fuchsian group $\widehat{\Delta}$ contining it as a finite index subgroup. The group 
$\Delta$ contains as a normal subgroup a Fuchsian group $\Gamma$ (of signature $(0;p,p,p,p)$) such that ${\mathbb H}^{2}/\Gamma=C_{\lambda}/H$ and, if 
$\Gamma'$ denotes the commutator subgroup of $\Gamma$, then $C_{\lambda}={\mathbb H}^{2}/\Gamma'$ and $H=\Gamma/\Gamma'$. There is also a normal subgroup $\Gamma_{S}$ of $\Gamma$ such that $\Gamma' \leq \Gamma_{S}$ ($K=\Gamma_{S}/\Gamma'$ and $G=\Gamma/\Gamma_{S}$). The group $\Gamma_{S}$ is a normal subgroup of $\Delta$ such that $\Delta/\Gamma_{S}={\rm Aut}_{G}(S)$. Now, if we assume $p \geq 7$, then ${\rm Aut}(S)={\rm Aut}_{G}(S)$ and, in this case, $\Gamma_{S}$ cannot be normal subgroup of $\widehat{\Delta}$. If $p=5$, the situation may change (in this case it might be that  ${\rm Aut}(S) \neq {\rm Aut}_{G}(S)$) as can be seen in \cite{BCI2}.
\end{rema}

\begin{rema}[Field of moduli versus field of definitions]
Let us assume $p \geq 7$, so ${\rm Aut}_{G}(S)={\rm Aut}(S)$.
In case \ref{caso1} with the signature $(0;p,2p,2p)$ and cases \ref{caso4}, \ref{caso5}, \ref{caso5} and \ref{caso6}, the curves are quasiplatonic, so they are definable over their fields of moduli \cite{Wolfart}.
In case \ref{caso3}, the curve has odd signature, so it is also defined over its field of moduli \cite{AQ}. In the cases \ref{caso1} and \ref{caso2} with the signature $(0;2,2,p,p)$
as $S/{\rm Aut}(S)$ has genus zero, there is a minimal field of definition for $S$ which is at most a quadratic extension over its field of moduli.
\end{rema}

\begin{example}\label{ejemplo2}
Let $p \geq 7$ be a prime integer and 
consider the subgroup $K=\langle a_{1}a_{2}, a_{1}a_{3}^{-1}\rangle \cong {\mathbb Z}_{p}^{2}$. This is the kernel of the surjective homomorphism $\eta:H \to {\mathbb Z}=\{0,1,\ldots,p-1\}$ defined by $\eta(a_{1})=\eta(a_{3})=1$ and $\eta(a_{2})=\eta(a_{4})=p-1$. It can be seen that $K$ acts freely on $C_{\lambda}$, so it defines a closed Riemann surface $S=C_{\lambda}/K$ of genus $g=p-1$ admitting $G=H/K \cong {\mathbb Z}_{p}$ as a group of conformal automorphisms such that $S/G=C_{\lambda}/H$, i.e., the Riemann sphere $\widehat{\mathbb C}$ and cone points $\infty, 0, 1$ and $\lambda$. 
Since $\Theta(\widehat{\alpha})=(1,2)(3,4)$ and $a_{2}a_{4}^{-1}= (a_{1}a_{2})(a_{2}a_{4}^{-1})(a_{1}a_{2})^{-1}$, we see that $K$ is invariant under $\widehat{\alpha}$.
Similarly, as  $\Theta(\widehat{\beta})=(1,3)(2,4)$ and $a_{3}a_{4}=(a_{1}a_{2})^{-1}$ and $a_{1}a_{3}^{-1}=(a_{1}a_{3})^{-1}$, we also have that $K$ is invariant under $\widehat{\beta}$.
In particular, ${\rm Aut}(S)$ contains the group $G \rtimes \langle a,b\rangle \cong {\mathbb Z}_{p} \rtimes {\mathbb Z}_{2}^{2}$, where $a$ (respectively, $b$) is induced by $\widehat{\alpha}$ (respectively, $\widehat{\beta}$). If $\lambda \neq 2$, then ${\rm Aut}(S)=G \rtimes \langle a,b\rangle$.
If $\lambda=2$, then $K$ is also invariant under $\widehat{\gamma}$, and ${\rm Aut}(S)=G \rtimes D_{4}$ (this corresponds to Accola-Maclachlan's curve). 
The above surfaces correspond to the algebraic curves of the form
$$y^{p}=x(x-1)(x-\lambda)^{p-1}.$$ 
The induced automorphisms on $S$ corresponding to $a$ and $b$ (in the the above algebraic model) is given by
$$(x,y) \stackrel{a}{\mapsto} \left(\frac{\lambda}{x},\frac{-\lambda (x-1)(x-\lambda)}{xy} \right), \; 
(x,y) \stackrel{b}{\mapsto} \left(\frac{x-\lambda}{x-1},\frac{-(1-\lambda) x(x-\lambda)}{(x-1)y} \right),$$
and, for $\lambda=2$, the induced automorphism $c$ by $\widehat{\gamma}$ is
$$(x,y) \stackrel{c}{\mapsto} \left(\frac{2}{2-x},\frac{-2y(x-1)}{(x-2)^{2}} \right)$$
\end{example}

\subsection{Case $l=2$: ${\mathbb Z}_{p}^{2}$ actions of type $4$}\label{Sec:caso:l=2}
In this case,  $K$ is one of the following $p^{2}+p-3$ possible subgroup (Section \ref{Sec:r=1}):
$$\langle a_{2}a_{3}^{k} \rangle, \; k=1,\ldots, p-1,$$
$$\langle a_{1}a_{2}^{r}a_{3}^{s} \rangle, \; (r,s) \in \{0,1,\ldots,p-1\} \setminus \{(0,0), (1,1)\}.$$

We have a surjective homomorphism $\varphi:H \to H/K=G=\langle A,B\rangle$, where $\varphi(a_{1})=A$ and $\varphi(a_{2})=B$. As $a_{1}$ and $a_{2}$ have both fixed points on $C_{\lambda}$, the induced automorphisms $A$ and $B$ also have fixed points on $S=C_{\lambda}/K$.

\begin{rema}[Non-hyperellipticity]
Note that $S$ cannot be hyperelliptic. If that was the case, the group $G \cong {\mathbb Z}_{p}^{2}$ must induce an isomorphic group (as the hyperelliptic involution does not belong to $G$) as a subgroup of ${\rm PSL}_{2}({\mathbb C})$, which is not possible for $p \neq 2$.
\end{rema}

\begin{rema}[Algebraic models] Let us assume $p \geq 5$ (the case $p=3$ will be observed later).
The group $G$ has $(p+1)$ cyclic subgroups: $\langle B \rangle$ and $\langle A B^{l}\rangle$, where $l=0,1,\ldots,p-1$. We already know that both $\langle B \rangle$ and $\langle A \rangle$ act with fixed points. Let us consider the rest of the cyclic subgroups $\langle A B^{l}\rangle$, where $l=1,\ldots,p-1$.
As $S/G$ has only four cone points, at most two of them may act with fixed points. 
Since $p-1>2$, we can find $l_{1} \neq l_{2}$ such that $\langle A B^{l_{1}}\rangle$ and $\langle A B^{l_{2}}\rangle$ both acts freely on $S$. 
If we set $A_{1}=A B^{l_{1}}$ and $B_{1}=A B^{l_{2}}$,  then
(a) $G=\langle A_{1}, B_{1}\rangle$ and (b) both, $S/\langle A \rangle$ and $S/\langle B \rangle$, are cyclic $p$-gonal curves, each one a cyclic branched covering of $\widehat{\mathbb C}$ (with branch values at $\infty, 0, 1$ and $\lambda$). It follows that $S$ can be algebraically described by the fiber product
$$ \; \left\{
\begin{array}{l}
y^{p}=x(x-1)^{a}(x-\lambda)^{b}\\
z^{p}=x(x-1)^{c}(x-\lambda)^{d}
\end{array}
\right\},
$$
where $a,b,c,d \in \{1,\ldots,p-1\}$ are such that $1+a+b,1+c+d \nequiv 0 \mod(p)$. In this (singular) algebraic model, $A_{1}(x,y,z)=(x,\omega_{p}y,z)$ and $B_{1}(x,y,z)=(x,y,\omega_{p}z)$.
\end{rema}

\subsubsection{\bf Case 1}
If $K=\langle a_{2}a_{3}^{k}\rangle$, then $\varphi(a_{3})=B^{\hat{k}}$ and $\varphi(a_{4})=A^{-1}B^{\hat{k}-1}$, where $k \hat{k} \equiv 1 \mod p$. This, in particular asserts that $0$ and $1$ are projection of fixed points of $B$ (so $B$ has exactly $2p$ fixed points), that $A$ has exactly $p$ fixed points, that $A^{-1}B^{\hat{k}-1}$ has also exactly $p$ fixed points and that every other element of $G$ not in $\langle A \rangle \cup \langle B \rangle \cup \langle A^{-1}B^{\hat{k}-1} \rangle$ has no fixed points. In this case, the quotient orbifold ${\mathcal O}=S/\langle B \rangle$ has genus zero (so we may assume it to be $\widehat{\mathbb C}$)  and exactly $2p$ cone points, each one of order $p$. Moreover, the element $A$ induces a conformal automorphism of order $p$ of ${\mathcal O}$. Without lost of generality, we may assume it to be the rotation $x \mapsto \omega_{p} x$ and the cone points to be given by the $p$-roots of unity and the $p$-roots of $\lambda$. 

\begin{rema}\label{remaE}
In this case, a much more simple algebraic model for $S$ (now including the case $p \geq 3$) is given by
$$E_{\lambda}: \; y^{p}=(x^{p}-1)(x^{p}-\lambda)^{m}, \; m \in \{1,\ldots,p-1\},$$
where $G=\langle A(x,y)=(\omega_{p} x, y), B(x,y)=(x,\omega_{p} y)\rangle$, 
and for which $\pi(x,y)=x^{p}$ is a regular branched cover with deck group $G$.
\end{rema}

\subsubsection{\bf Case 2}
If $K=\langle a_{1}a_{2}^{r}a_{3}^{s}\rangle$, where either $r=0$ or $s=0$ (but not both), then the situation is similar to the previous casse.

\subsubsection{\bf Case 3}
If $K=\langle a_{1}a_{2}^{r}a_{3}^{s}\rangle$, where $r,s>0$, then $\varphi(a_{3})=A^{-\hat{s}}B^{-\hat{s}r}$ and $\varphi(a_{4})=A^{\hat{s}-1}B^{\hat{s}r-1}$, where $s \hat{s} \equiv 1 \mod p$. This, in particular asserts that each $A$, $B$, $AB^{r}$ and $A^{\hat{s}-1}B^{\hat{s}r-1}$ has exactly $p$ fixed points and 
 every other element of $G$ not in $\langle A \rangle \cup \langle B \rangle \cup \langle AB^{r} \rangle \cup \langle A^{\hat{s}-1}B^{\hat{s}r-1}  \rangle$ has no fixed points.
In this case, the quotient orbifold ${\mathcal O}=S/\langle A \rangle$ has genus $(p-1)/2$  and exactly $p$ cone points, each one of order $p$. Moreover, the element $B$ induces a conformal automorphism of order $p$ of ${\mathcal O}$ with exactly three fixed points. We may assume that there is a $p$-fold regular branched cover from ${\mathcal O} \to \widehat{\mathbb C}$ whose branch values are $0$, $1$ and $\lambda$.

\begin{rema}
In this situation, for $p=3$, we may observe that an algebraic model for $S$ is given as
$$ \; \left\{
\begin{array}{l}
y^{3}=x(x-1)^{a}(x-\lambda)^{b}\\
z^{3}=(x-1)^{c}(x-\lambda)^{d}
\end{array}
\right\},
$$
where $a,b,c,d \in \{1,2\}$ are such that $a+b \equiv 2 \mod 3$ and $c+d \nequiv 0 \mod(3)$. In this (singular) algebraic model, $A(x,y,z)=(x,\omega_{3}y,z)$ and $B(x,y,z)=(x,y,\omega_{3}z)$.
\end{rema}

\subsubsection{\bf A description of ${\rm Aut}(S)$ in terms of $K$}
Recall that, for $p \geq 13$, we have the equality ${\rm Aut}_{G}(S)={\rm Aut}(S)$.

\begin{theo}\label{teol=2}
Following with the previous notations.
\begin{enumerate}
\item If either $K=\langle a_{2}a_{3}^{k}\rangle$, $k \neq 1$, or $K=\langle a_{1}a_{3}^{s}\rangle$ or $K=\langle a_{1}a_{2}^{r}\rangle$, where $r,s \notin \{1,p-1\}$,
then ${\rm Aut}_{G}(S)=G \cong {\mathbb Z}_{p}^{2}$. In this case, $S/{\rm Aut}_{G}(S)$ has signature $(0;p,p,p,p)$.

\item If $K=\langle a_{2}a_{3}\rangle$, then ${\rm Aut}_{G}(S)=G \rtimes  {\mathbb Z}_{2}^{2}$.
In this case, $S/{\rm Aut}_{G}(S)$ has signature $(0;2,2,2,p)$ and, in particular, $S$ is definable over its field of moduli. 

\item If $K=\langle a_{1}a_{3}\rangle$ and $\lambda=2$, then ${\rm Aut}_{G}(S)=G \rtimes {\mathbb Z}_{4}$.
In this case, $S/{\rm Aut}_{G}(S)$ has signature $(0;4,4,p)$ and, in particular, $S$ is definable over its field of moduli. 

\item If either $K=\langle a_{1}a_{3}^{-1}\rangle$ or $K=\langle a_{1}a_{2}^{\pm 1}\rangle$,
then ${\rm Aut}_{G}(S)=G \rtimes {\mathbb Z}_{2}$. In this case, $S/{\rm Aut}_{G}(S)$ has signature $(0;2,2,p,p)$.  

\item If $K=\langle a_{1}a_{2}^{r}\rangle$, then (i) for $r \notin\{1, p-1\}$, ${\rm Aut}_{G}(S)=G$, (ii) for $r=1$, ${\rm Aut}_{G}(S)=G \rtimes {\mathbb Z}_{2}^{2}$ (in which case $S/{\rm Aut}_{G}(S)$ has signature $(0;2,2,2,p)$ and (iii) for $r=p-1$, ${\rm Aut}_{G}(S)=G \rtimes {\mathbb Z}_{2}$ (in which case $S/{\rm Aut}_{G}(S)$ has signature $(0;2,2,p,p)$.

\item If $K=\langle a_{1}a_{3}^{s}\rangle$, then (i) for $s \neq 1$, ${\rm Aut}_{G}(S)=G$, and (ii) for $s=1$, ${\rm Aut}_{G}(S)=G \rtimes {\mathbb Z}_{2}$ (in which case $S/{\rm Aut}_{G}(S)$ has signature $(0;2,2,p,p)$.

\item If $K=\langle a_{1}a_{2}^{r}a_{3}^{s} \rangle$, where $r,s>0$, then ${\rm Aut}_{G}(S)=G \cong {\mathbb Z}_{p}^{2}$ (in which case, $S/{\rm Aut}_{G}(S)$ has signature $(0;p,p,p,p)$) with the only following exceptions:
\begin{enumerate}
\item (i) $r=s-1$, $s\in \{2,\ldots,p-1\}$ or (ii) $r=s+1$, $s \in \{1,\ldots,p-2\}$ or (iii) $r=p+1-s$, $s\in\{2,\ldots,p-1\}$,
in which case, ${\rm Aut}_{G}(S)=G \rtimes {\mathbb Z}_{2}$ and $S/{\rm Aut}_{G}(S)$ has signature $(0;2,2,p,p)$;

\item $s^{2}-3s+3 \equiv 0 \mod p$ and $r \equiv 2s-s^{2} \mod p$, where ${\rm Aut}_{G}(S)=G \rtimes {\mathbb Z}_{3}$ (in which  case, $S/{\rm Aut}_{G}(S)$ has signature $(0;3,p,3p)$ and $S$ is definable over its field of moduli); and
\item $s^{3}-s^{2}+s-1 \equiv 0 \mod p$ and $r \equiv s-s^{2} \mod p$, where ${\rm Aut}_{G}(S)=G \rtimes {\mathbb Z}_{4}$ (in which case, $S/{\rm Aut}_{G}(S)$ has signature $(0;4,4,p)$  and $S$ is gain definable over its field of moduli).
\end{enumerate}
\end{enumerate}

\end{theo}
\begin{proof}
{\bf (1) Case $K=\langle a_{2}a_{3}^{k} \rangle$.}
As $\widehat{\alpha}$ conjugates $a_{2}a_{3}^{k}$ to $a_{1}a_{4}^{k}$, and the last element belong to $K$ only for $k=1$ (this can be seen by looking at the diagonal forms of the standard generators), it follows that this automorphism does not keep invariant (under conjugation) $K$ with the only exception $k=1$. 
Similarly, $\widehat{\beta}$ conjugates $a_{2}a_{3}^{k}$ to $a_{4}a_{1}^{k}$ and this belong to $K$ only for $k=1$.
The element $\widehat{\delta}$ (respectively, $\widehat{\gamma}$) conjugates $a_{2}a_{3}^{k}$ to $a_{3}a_{4}^{k} \notin K$.

{\bf (2) Case $K=\langle a_{1}a_{2}^{r}a_{3}^{s} \rangle$.} 
The element $\widehat{\alpha}$ conjugates $a_{1}a_{2}^{r}a_{3}^{s}$ to $a_{1}^{r}a_{2}a_{4}^{s}$, which belongs to $K$ only for (a) $r \in \{1,p-1\}$ and $s=0$ or (b) $r=s-1$ and $s \geq 1$. In fact, for $a_{1}^{r}a_{2}a_{4}^{s}$, to belongs to $K$, we must be some $l \in \{1,\ldots,p-1\}$ such that 
$a_{1}^{r}a_{2}a_{4}^{s}=(a_{1}a_{2}^{r}a_{3}^{s})^{l}$, that is,
(i) $\omega_{p}^{l}=\omega_{p}^{r-s}$, (ii) $\omega_{p}^{lr}=\omega_{p}^{1-s}$ and (iii) $\omega_{p}^{ls}=\omega_{p}^{-s}$. It follows from (iii) that either $l=p-1$ or $s=0$.
If $s=0$ (so $r \neq 0$), then (i) asserts that $l \equiv r \mod p$, i.e., $l=r$ (since $1 \leq l,r \leq p-1$) and by (ii) we get that $r^{2}-1$ is divisible by $p$, so either $r=1$ or $r=p-1$. If $s>0$, then $l=p-1$. By (i) it follows that $r-s+1 \equiv 0 \mod p$, i.e., $r=s-1$.
The element $\widehat{\beta}$ conjugates $a_{1}a_{2}^{r}a_{3}^{s}$ to $a_{1}^{s}a_{3}a_{4}^{r}$ and similar computations as above permits to see that $a_{1}^{s}a_{3}a_{4}^{r}$ belongs to $K$ only for $0 \leq s \leq p-2$ and $r=s+1$. 
The element $\widehat{\beta}\;\widehat{\alpha}$ conjugates $a_{1}a_{2}^{r}a_{3}^{s}$ to $a_{2}^{s}a_{3}^{r}a_{4}$ which belongs to $K$ only for $s+r \equiv 0 \mod p$.
The element $\widehat{\delta}$ conjugates $a_{1}a_{2}^{r}a_{3}^{s}$ to $a_{1}a_{3}^{r}a_{4}^{s}$. 
If $r=0$ or $s=0$, then $a_{1}a_{3}^{r}a_{4}^{s} \notin K$. Now, let us assume $r,s>0$.
For the element $a_{1}a_{3}^{r}a_{4}^{s}$ to belong to $K$ there must be some $l \in \{1,\ldots,p-1\}$ such that (i) $l \equiv 1-s \mod p$, (ii) $lr \equiv -s \mod p$ and (iii) $ls \equiv r-s \mod p$. In particular, $s \neq 1$ (otherwise, it obligates for $l=0$), i.e., $s \in \{2,\ldots,p-1\}$ and (i) ensures that $l=p+1-s$. The other two conditions (ii) and (iii) enusre 
that both $(1-s)r+s$ and $(1-s)s+s-r$ are divisible by $p$. Note that neither $r=1$.
From the last one, $r=(1-s)s+s+qp$, some $q$, and replacing it in the second one, we obtain that $s(s^{2}-3s+3)$ is divisible by $p$, i.e., $s^{2}-3s+3$ is divisible by $p$. 
The element $\widehat{\gamma}$ conjugates $a_{1}a_{2}^{r}a_{3}^{s}$ to $a_{2}a_{3}^{r}a_{4}^{s}$. If $r=0$, then for the previous element to stay in $K$ it is needed that $s=1$. Now assume $r>0$. If $s=0$, then $a_{2}a_{3}^{r} \notin K$. So, we also assume $s>0$. For the element $a_{2}a_{3}^{r}a_{4}^{s}$ to belong to $K$ there must be some $l \in \{1,\ldots,p-1\}$ such that (i) $l \equiv -s \mod p$, (ii) $lr \equiv 1-s \mod p$ and (iii) $ls \equiv r-s \mod p$. Form (i) we must have $l=p-s$. These equations asserts that $s^{3}-s^{2}+s-1 \equiv 0 \mod p$ and $r \equiv s-s^{2} \mod p$.
\end{proof}

\begin{rema}
As was the case for the case $l=1$, in the above theorem we obtain (in many cases) that the signature of $S/{\rm Aut}_{G}(S)$ is non-maximal. Now, in the case that $p \geq 13$, we have that ${\rm Aut}_{G}(S)={\rm Aut}(S)$ and we have a similar observation as done in Remark \ref{obsnomax}. Situation might be quite different for $p \in \{3,5,7,11\}$.

\end{rema}

\begin{example}
Let $p \geq 13$ be a prime integer.
Some axamples of algebraic curves representing Riemann surfaces $S$ admitting a group $G \cong {\mathbb Z}_{p}^{2}$ such that $S/G$ has signature $(0;p,p,p,p)$ are provided by  the curves $E_{\lambda}$ of Remark \ref{remaE}. 
\begin{enumerate}
\item If $m=p-1$, then $E_{\lambda}$ admits the extra commuting automorphisms of order two 
$$\alpha(x,y)=\left( \frac{\sqrt[p]{\lambda}}{x}, \sqrt[p]{-\lambda^{p-1}} \; y^{p-1}\right), \; 
\beta(x,y)=\left( \frac{x^{p}-\lambda}{y}, (1-\lambda) x^{p-1}\right).$$

In this case, ${\rm Aut}(S)=G \rtimes \langle \alpha,\beta\rangle \cong {\mathbb Z}_{p}^{2} \rtimes {\mathbb Z}_{2}^{2}$ and $S/{\rm Aut}(S)$ has signature $(0;2,2,2,p)$. This corresponds to the group $K=\langle a_{2}a_{3}\rangle$.

\item If $m^{2}+1 \equiv 0 \mod p$, then $E_{2}$ admits the order four automorphism
$$\gamma(x,y)=\left( \frac{(x^{p}-1)(x^{p}-2)^{m-1}}{y^p}, \sqrt[p]{(-1)^{(p+1)/2}2^{(p-1)/2}} \; (x^{p}-2)^{(m^{2}+1)/p}\right).$$

In this case, ${\rm Aut}(S)=G \rtimes \langle \gamma\rangle \cong {\mathbb Z}_{p}^{2} \rtimes {\mathbb Z}_{4}$ and $S/{\rm Aut}(S)$ has signature $(0;4,4,p)$. This case corresponds to the group $K=\langle a_{1}a_{3}\rangle$.

\item If $m=1$, then $E_{\lambda}$ admits the order two automorphism
$$\alpha(x,y)=\left( \frac{\sqrt[p]{\lambda}}{x}, \sqrt[p]{\lambda} \; y\right).$$

In this case, ${\rm Aut}(S)=G \rtimes \langle \alpha\rangle \cong {\mathbb Z}_{p}^{2} \rtimes {\mathbb Z}_{2}$ and 
$S/{\rm Aut}(S)$ has signature $(0;2,2,p,p)$. This corresponds to the groups 
$K=\langle a_{1}a_{3}^{-1}\rangle$ or $K=\langle a_{1}a_{2}^{\pm 1}\rangle$.

\end{enumerate}

\end{example}

\section{An isogenous decomposition of the jacobian variety}
Let $S$ be a closed Riemann surface of genus $g \geq 2$. The jacobian variety of $S$ is given by the complex torus  $JS=H^{1,0}(S)^{\vee}/H_{1}(S,{\mathbb Z})$, where $H^{1,0}(S)^{\vee}$ is the dual space of the $g$-dimensional space $H^{1,0}(S)$ of holomorphic one-forms, and $H_{1}(S,{\mathbb Z})$ is the lattice generated by the integration of forms along homology curves. The intersection form in homology induces a principal polarization of $JS$, that is, $JS$ has the structure of a principally polarized abelian variety \cite{BL}. 
Two principally polarized abelian varieties $A_{1}$ and $A_{2}$ are called isogenous if there is a 
non-constant surjective morphism $h:A_{1} \to A_{2}$ between the corresponding tori with finite kernel ($h$  
is called an isogeny). A principally polarized abelian variety $A$ is called decomposable if it is isogenous to the product of abelian varieties of smaller dimensions (otherwise, it is said to be simple). As a consequence of Poincar\'e complete reducibility theorem \cite{Poincare}, every principally polarized abelian variety can be decompososed as the product of (simple) ones. 
In particular, the Jacobian variety $JS$ can be decomposed, up to isogeny, into a product of simple sub-varieties. 

If we are given a group $G$ of conformal automorphisms of $S$, in \cite{LR}, it has been obtained a method to obtain an isogenous decomposition of $JS$. In such a decomposition, some of the factor may be jacobian varieties and others might not.
One may wonder for an isogenous decompositions of $JS$ such that each of the factors is the jacobian variety of some closed Riemann surface.
The following result, due to Kani and Rosen \cite{K-R}, provides sufficient conditions for it to happens. If $L<{\rm Aut}(S)$, we denote by $g_{L}$ the genus of the quotient orbifold $S/L$ and by $S_{L}$ its the underlying Riemann surface structure.

\begin{theo}[Corollary of Theorem C in \cite{K-R}]\label{coroKR}
Let $S$ be a closed Riemann surface of genus $g \geq 1$ and let $H_{1},\ldots,H_{t}<{\rm Aut}(S)$ such that:
\begin{enumerate}
\item $H_{i} H_{j}=H_{j} H_{i}$, for all $i,j =1,\ldots,t$;
\item $g_{H_{i}H_{j}}=0$, for $1 \leq i < j \leq t$
\item $g=\sum_{j=1}^{t} g_{H_{j}}$.
\end{enumerate}
Then 
$$JS \cong_{isog.}  \prod_{j=1}^{t} JS_{H_{j}}.$$
\end{theo}

\subsection{Decomposition for generalized Fermat curves}
Let $\Lambda=(\lambda_{1},\ldots,\lambda_{n-2}) \in \Omega_{n}$ and consider the generalized Fermat curve $C_{\Lambda}$ together its generalized Fermat group $H \cong {\mathbb Z}_{p}^{n}$, where $p \geq 2$ is a prime integer such that $(p-1)(n-1)>2$. As a consequence of Kani-Rosen result, the following isgenous decomposition is known.

\begin{theo}[\cite{CHQ}]\label{(p,n)}
$$JC_{\Lambda} \cong_{isog.} \prod_{H_r} JS_{H_r},$$
where $H_{r}$ runs over all subgroups of $H$ which are isomorphic to ${\mathbb Z}_{p}^{n-1}$, with quotient orbifold $S/H_{r}$ of genus at least one, and $S_{H_r}$ denoting the underlying Riemann surface of the orbifold.
The cyclic $p$-gonal curves $S_{H_r}$ run over all curves of the form
$$y^{p}=\prod_{j=1}^{r}(x-\mu_{j})^{\alpha_{j}},$$
where $\{\mu_{1},\ldots,\mu_{r}\} \subset \{\infty,0,1,\lambda_{1},\ldots,\lambda_{n-2}\}$, $\mu_{i} \neq \mu_{j}$ if $i \neq j$,  $\alpha_{j} \in \{1,2,\ldots,p-1\}$ satisfying the following.
\begin{enumerate}
\item[(i)] If every $\mu_{j} \neq \infty$, then $\alpha_{1}=1$, $\alpha_{2}+\cdots+\alpha_{r} \equiv p-1 \mod(p)$;
\item[(ii)] If some $\lambda_{a}=\infty$, then  $\alpha_{1}+\cdots+\alpha_{a-1}+\alpha_{a+1}+\cdots +\alpha_{r} \equiv p-1 \mod(p)$.
\end{enumerate}  
\end{theo}

\subsection{Decomposition for the general case}
Let us now consider a pair $(S,G)$, where $S$ is a closed Riemann surface and $G \cong {\mathbb Z}_{p}^{l}$ of type $n+1 \geq 3$, where $p \geq 2$ is a prime integer such that $(p-1)(n-1)>2$. As already seen, there is some  $\Lambda=(\lambda_{1},\ldots,\lambda_{n-2}) \in \Omega_{n}$ and a subgroup $K \cong {\mathbb Z}_{p}^{n-l}$ of the generalized Fermat group $H$ of $C_{\Lambda}$, acting freely on it, such that there is a biholomorphism $\phi:S \to C_{\Lambda}/K$ conjugating $G$ to $H/K$.

There is a natural isogenety between $JC_{\Lambda}$ and $JS \times {\rm Prym}(C_{\Lambda}/S)$, where ${\rm Prym}(C_{\Lambda}/S)$ is the Prym variety associated to the covering $C_{\Lambda} \to S$ (this subvariety can be seen as the component of the identity of the ortogonal complement of the embedding of $JS$ into $JC_{\Lambda}$). On the other hand, Theorem \ref{(p,n)} asserts that $JC_{\Lambda}$ is isogenous to a product of certain jacobian varieties. One may conjecture that the above induces an isogenous decomposition

$$JS \cong_{isog.} \prod_{Z} JS_{Z},$$
where $Z$ runs over all subgroups of $G$ which are isomorphic to ${\mathbb Z}_{p}^{l-1}$, with quotient orbifold $S/Z$ of genus at least one (where $S_{Z}$ denotes its underlying Riemann surface), and that 
$${\rm Prym}(C_{\Lambda}/S) \cong_{isog.} \prod_{H_{r}} JS_{H_{r}},$$
where $H_{r}$ runs over all subgroups of $H$ which are isomorphic to ${\mathbb Z}_{p}^{n-1}$, with quotient orbifold $S/H_{r}$ of genus at least one, and such that $H_{r}$ does not contains $K$.

In the following, we prove the above assertions for the case $l=2$ and type $4$.

\begin{theo}\label{isogeno}
Let $S$ be a closed Riemann surface and ${\mathbb Z}_{p}^{2} \cong G<{\rm Aut}(S)$ of type $4$, where $p \geq 3$ is a prime integer. Let ${\mathcal C}$ be the collection of order $p$ cyclic subgroups $Z$ of $G$ such that $S/Z$ has genus bigger than zero. Then 
$$JS \cong_{isog} \prod_{Z \in {\mathcal C}} JS_{Z}.$$
\end{theo}
\begin{proof}
By Theorem \ref{teo1}, we may assume that $S=C_{\lambda}/K$, where $K \cong {\mathbb Z}_{p}$ is a subgroup of $H$ acting freely on $C_{\lambda}$, and 
$G=\langle A, B\rangle=H/K$.

The group $G$ has exactly $p+1$ different cyclic subgroups $Z \cong {\mathbb Z}_{p}$. Let $S_{Z}$ be the underlying Riemann surface of the quotient orbifold $S/Z$. As a consequence of the Riemann-Hurwitz formula), either one of the following possibilities occur.
\begin{enumerate}
\item $Z$ acts freely on $S$ and $S_{Z}$ has genus $p-1$. We denote by ${\mathcal A}$ the collection of those subgroups $Z$.
\item $Z$ has exactly $p$ fixed point on $S$ and $S_{Z}$ has genus $(p-1)/2$. We denote by ${\mathcal B}$ the collection of those subgroups $Z$.
\item $Z$ has exactly $2p$ fixed points and $S_{Z}$ is of genus zero. We denote by ${\mathcal C}$ the collection of those subgroups $Z$.
\end{enumerate}

If $K=\langle a_{2}a_{3}^{k}\rangle$, then $\langle B\rangle$ belongs to ${\mathcal C}$, $\langle A\rangle$ and $\langle A^{-1}B^{\hat{k}-1}\rangle$ belong to ${\mathcal B}$ and all other cyclic subgroups are in ${\mathcal A}$. In the case that $K=\langle a_{1}a_{2}^{r}a_{3}^{s}\rangle$, then $\langle A\rangle$, $\langle B\rangle$ and $\langle A^{\hat{s}-1}B^{\hat{s}r-1}\rangle$ belong to ${\mathcal B}$ and all other cyclic subgroups are in ${\mathcal A}$.

Let $\alpha, \beta, \gamma \geq 0$ be the cardinalities of ${\mathcal A}$, ${\mathcal B}$ and ${\mathcal C}$, respectively. So, $\alpha+\beta+\gamma=p+1$. As $S/G$ has exactly $4$ cone points, we observe that the total number of points on $S$ with non-trivial stabilizer is $4p$. It follows that $4p=p\beta+2p\gamma$. All the above asserts that $\alpha=\gamma+p-3$ and $\beta=4-2\gamma$. This permits to observe that $(p-1)^{2}=\alpha (p-1)+\beta(p-1)/2$. As $G$ is an abelian group, we may apply the Kani-Rosen theorem to the collection of those $\alpha+\beta$ cyclic subgroups $K$ as above to obtain that 
$$JS \cong_{isog} \prod_{Z \in {\mathcal A} \cup {\mathcal B}} JS_{Z}.$$

\end{proof}

\begin{rema}
In the above proof, 
if $Z \in {\mathcal A}$, then $S_{Z}$ admits $G/Z \cong {\mathbb Z}_{p}$ as a group of conformal automorphisms such that $S_{Z}/(G/Z)$ is $S/G$. It follows that $S_{Z}$ can be described by a cyclic $p$-gonal curve of the form $y^{p}=x(x-1)^{a_{Z}}(x-\lambda)^{b_{Z}}$, where $a_{Z},b_{Z} \in \{1,\ldots, p-1\}$ are such that $1+a_{Z}+b_{Z} \nequiv 0 \mod p$.
If $Z \in {\mathcal B}$, then $S_{Z}$ admits $G/Z \cong {\mathbb Z}_{p}$ as a group of conformal automorphisms such that $S_{Z}/(G/Z)$ is the Riemann sphere and its three cone points are contained inside the set $\{\infty,0,1,\lambda\}$. If $r_{Z},s_{Z},t_{Z}$ are these cone points, then 
$S_{Z}$ can be described by a cyclic $p$-gonal curve of the form $y^{p}=(x-r_{Z})(x-s_{Z})^{a_{Z}}(x-t_{Z})^{b_{Z}}$, where $a_{Z},b_{Z} \in \{1,\ldots, p-1\}$ are such that $1+a_{Z}+b_{Z} \equiv 0 \mod p$ (where if one of the cone points is $\infty$, the corresponding factor is deleted).

\end{rema}

\section{Declarations}

Funding: FONDECYT, ANID, Chile 

Conflicts of interest/Competing interests:   Not applicable

Availability of data and material: Not applicable

Code availability: Not applicable


\end{document}